\RequirePackage{silence}
\documentclass[11pt,english]{article}

\usepackage[margin= 2.2 cm, bottom=19.5mm,footskip=12mm,top=20mm]{geometry}

\usepackage{amsthm}
\usepackage{amsmath}
\usepackage{amssymb}
\usepackage{setspace}
\usepackage{mathtools}
\usepackage{graphicx}
\usepackage[hidelinks]{hyperref}
\usepackage{thm-restate}
\usepackage{cleveref}
\usepackage{graphicx}
\usepackage{enumitem}
\usepackage{framed}
\usepackage{subcaption}

\usepackage{floatrow}
\usepackage[T1]{fontenc}
\usepackage{tikz}
\usepackage{mathdots}
\usepackage{xcolor}
\usepackage{diagbox}
\usepackage{colortbl}
\usepackage[absolute,overlay]{textpos}
\usetikzlibrary{calc}
\usetikzlibrary{decorations.pathreplacing}
\usetikzlibrary{positioning,patterns}
\usetikzlibrary{arrows,shapes,positioning}
\usetikzlibrary{decorations.markings}

\tikzstyle{edge}=[very thick]
\definecolor{bostonuniversityred}{rgb}{0.8, 0.0, 0.0}
\definecolor{arsenic}{rgb}{0.23, 0.27, 0.29}
\tikzstyle{diredge}=[postaction={decorate,decoration={markings,
		mark=at position .95 with {\arrow[scale = 1]{stealth};}}}]

\newcommand{\defPt}[3]{
	\def \pt {(#1, #2)}
	\coordinate [at = \pt, name = #3];
}

\newcommand{\defPtm}[2]{

	\coordinate [at = #1, name = #2];
}

\newcommand{\fitellipsis}[2] 
{\draw [fill=green]let \p1=(#1), \p2=(#2), \n1={atan2(\y2-\y1,\x2-\x1)}, \n2={veclen(\y2-\y1,\x2-\x1)}
    in ($ (\p1)!0.5!(\p2) $) ellipse [ x radius=\n2/2+0cm, y radius=0.1cm, rotate=\n1];
}

\theoremstyle{plain}

\newtheorem*{thm*}{Theorem}
\newtheorem{thm}{Theorem}
\Crefname{thm}{Theorem}{Theorems}
\numberwithin{thm}{section}

\newtheorem*{lem*}{Lemma}
\newtheorem{lem}[thm]{Lemma}
\Crefname{lem}{Lemma}{Lemmas}

\newtheorem*{claim*}{Claim}

\crefname{claim}{Claim}{Claims}
\Crefname{claim}{Claim}{Claims}

\Crefname{prop}{Proposition}{Propositions}

\newtheorem{cor}[thm]{Corollary}
\crefname{cor}{Corollary}{Corollaries}

\crefname{conj}{Conjecture}{Conjectures}

\Crefname{qn}{Question}{Questions}

\newtheorem{obs}[thm]{Observation}
\Crefname{obs}{Observation}{Observations}

\newtheorem{ex}[thm]{Example}
\Crefname{ex}{Example}{Examples}

\theoremstyle{definition}

\Crefname{prob}{Problem}{Problems}

\Crefname{defn}{Definition}{Definitions}

\newtheorem*{defn*}{Definition}

\theoremstyle{remark}
\newtheorem*{rem}{Remark}

\renewenvironment{proof}[1][]{\begin{trivlist}
\item[\hspace{\labelsep}{\bf\noindent Proof#1.\/}] }{\qed\end{trivlist}}

\newcommand{\floor}[1]{
    \left\lfloor #1 \right\rfloor
}

\newcommand{\eps}{\varepsilon}

\expandafter\def\expandafter\normalsize\expandafter{%
    \normalsize
    \setlength\abovedisplayskip{4pt}
    \setlength\belowdisplayskip{4pt}
    \setlength\abovedisplayshortskip{4pt}
    \setlength\belowdisplayshortskip{4pt}
}

\usepackage[square,sort,comma,numbers]{natbib}
\newcommand{\HH}{\mathcal{H}}

\newcommand{\OO}{\mathcal{O}}
\newcommand{\G}{\mathcal{G}}

\renewcommand{\ex}{ex}

\title{\vspace{-0.8cm} Universal and unavoidable graphs}

\author{
Matija Buci\'c\thanks{Department of Mathematics, ETH, Z\"urich, Switzerland. Email: \href{mailto:matija.bucic@math.ethz.ch} {\nolinkurl{matija.bucic@math.ethz.ch}}.}
\and
Nemanja Dragani\'c \thanks{Department of Mathematics, ETH, Z\"urich, Switzerland. Email: \href{mailto:nemanja.draganic@math.ethz.ch} {\nolinkurl{nemanja.draganic@math.ethz.ch}}.}
\and
Benny Sudakov\thanks{Department of Mathematics, ETH, Z\"urich, Switzerland. Email:
\href{mailto:benjamin.sudakov@math.ethz.ch} {\nolinkurl{benjamin.sudakov@math.ethz.ch}}.
Research supported in part by SNSF grant 200021\_196965.}
}
 \date{}

\begin{document}

\maketitle

\begin{abstract}
   The Tur\'an number $\ex(n,H)$ of a graph $H$ is the maximal number of edges in an $H$-free graph on $n$ vertices. In $1983$ Chung and Erd\H{o}s asked which graphs $H$ with $e$ edges minimize $\ex(n,H)$. They resolved this question asymptotically for most of the range of $e$ and asked to complete the picture. In this paper we answer their question by resolving all remaining cases. Our result translates directly to the setting of universality, a well-studied notion of finding graphs which contain every graph belonging to a certain family. In this setting we extend previous work done by Babai, Chung, Erd\H{o}s, Graham and Spencer, and by Alon and Asodi. 
\end{abstract}
\section{Introduction}

\par The following question of Tur\'an dating back to 1941 \cite{Turan1941external} is one of the most classical problems of graph theory. Given a fixed graph $H,$ what is the maximal number of edges one can have in an $n$ vertex graph which does not contain a copy of $H$ as a subgraph? The answer to this question, denoted $\ex(n,H),$ is called the Tur\'an number of $H$. Tur\'an numbers have been extensively studied in the last 70 years, see for example the surveys \cite{sidorenko1995we,keevash2011hypergraph,furedi1991Turan,benny-survey,furedi-survey}.

Turan's problem leads to another very natural extremal question -- what is the largest size of a graph which we can not avoid as a subgraph in any graph on $n$ vertices and $e$ edges? In other words, what kind of a graph $H$ with a fixed number of edges (possibly depending on $n$) has minimal Tur\'an number?    

\par This question was first asked by Chung and Erd\H{o}s \cite{chung1983unavoidable} in 1983 and many questions of a similar flavour were later considered in  \cite{chung1987unavoidable,chung1983unavoidableS,chung1981minimal,duke1977systems}. Some of these questions have also featured in an Erd\H{o}s open problem collection \cite{chung1997open}. 

In this paper we revisit the original question of Chung and Erd\H{o}s. They say a graph $H$ is \emph{$(n,e)$-unavoidable} if every graph on $n$ vertices and $e$ edges contains a copy of $H$ as a subgraph. Let $f(n,e)$ be the maximum number of edges in an $(n,e)$-unavoidable graph. Chung and Erd\H{o}s obtain the following bounds on $f(n,e)$:
\vspace{-0.2 cm}
\begin{enumerate}[label=(\roman*)]
    \item \label{itm:3} $f(n,e)=\Theta\left(\lceil e/n \rceil^2\right)$ if $e \le n^{4/3}$ 
    \item \label{itm:4} $f(n,e)=\Theta\left(\frac{\sqrt{e}\log n}{\log \left(\binom{n}{2}/e\right)}\right)$ if $cn^{4/3} < e < \binom{n}{2}-n^{1+c}$ for any $0<c<1$.
    \item \label{itm:7} 
    $
    \Omega\left(\frac{m^2}{\log^2 m} \right)\leq  \binom{n}{2}-f\left(n,e\right)\leq \OO\left(\frac{m^2\log\log m}{\log m} \right)$ where $m=\binom{n}{2}-e,$ and $\binom{n}{2}-cn<e,$\\ for some $c>0$.
    
\end{enumerate}

In fact for certain regimes from parts (i) and (ii) they obtain even more precise bounds, some of which appear in a later paper \cite{chung1987unavoidable}. However, there is a gap between regimes of part (ii) and (iii) for which their methods fail to give an answer. Given this, they naturally ask what is the correct behaviour of $f(n,e)$ towards the end of the range, so when $e>\binom{n}{2}-n^{1+o(1)}$. In this paper we resolve this question and determine $f(n,e)$ for the remaining values of $e$.

\begin{thm}\label{thm:main}
For any $\eps>0$, if we let $m= \binom{n}{2}-e$  
$$f\left(n,e\right) = \begin{cases} \binom{n}{2}-\Theta\left(\frac{m^2}{\log^2 m}\right)  & \text{ if } m \leq n \log n \\
\Theta\left(\frac{n^3 \log n}{m}\right) & \text{ if } n \log n < m < n^{3/2-\eps}
\end{cases}$$
\end{thm}
It is worth noting that the unavoidable graph we use in order to obtain the lower bound in the second part of the above theorem is the random Erd\H{o}s-Renyi graph. This is in stark contrast to the very structured graph used by Chung and Erd\H{o}s in regimes of (ii)\footnote{Note that the bound of Chung and Erd\H{o}s agrees, after some simplification, with our bound in \Cref{thm:main}, for the part of the range where they overlap.} and (iii) above. In particular, they use a disjoint union of complete bipartite graphs. Furthermore, for part of the regime both of these very different examples are extremal, up to a constant factor.

\subsection{Universality}
An $(n,e)$-unavoidable graph $H$ on $n$ vertices is contained in every graph $G$ on $n$ vertices and $e$ edges. Another way of saying this is that the complement of $H$ contains the complement of $G$. Since $G$ was arbitrary this means that $H$ is $(n,e)$-unavoidable if and only if its complement contains every graph on $n$ vertices and $m=\binom{n}{2}-e$ edges as a subgraph. This observation, made by Chung and Erd\H{o}s, links unavoidability to perhaps an even more natural and well studied notion, namely that of universality.

For a given family of graphs $\HH$ we call a graph $\HH$-universal if it contains a copy of each graph in $\HH$. Given $\HH$ one usually wants to find the smallest $\HH$-universal graph, with respect to the number of edges or both vertices and edges.
Chung and Graham \cite{chung1979universal} were the first to use this general notion of universality in 1979. In \cite{chung1979universal} they survey a number of results in this setting. Universality problems have been extensively studied ever since, for some examples see \cite{chung1983universal,conlon2017almost,montgomery2019spanning,alon2007embedding,bhatt1989universal} and references therein. 

Following the above observation by Chung and Erd\H{o}s the relevant family in our case is that of all graphs on $n$ vertices and $m$ edges, which we denote as $\HH(n,m)$. We denote by $g(n,m)$ the minimum number of edges in an $\HH(n,m)$-universal graph on $n$ vertices. The above observation of Chung and Erd\H{o}s boils down to the following relation between the functions $f$ and $g$:
\begin{equation}\label{eq:relation}
     g(n,m)=\binom{n}{2}-f\left(n,e\right)\qquad  
\end{equation}
where $m=\binom{n}{2}-e.$ This relation allows us to easily translate results between unavoidability and universality. For example, Theorem \ref{thm:main} is equivalent to the following statement.

\begin{thm}\label{thm:main2}
For any $\eps>0$  we have 
$$g(n,m) = \begin{cases} \Theta\left(\frac{m^2}{\log^2 m}\right)  & \text{ if } m \leq  n \log n \\
\binom{n}{2}-\Theta\left(\frac{n^3 \log n}{m}\right) & \text{ if } n \log n < m < n^{3/2-\eps}
\end{cases}$$
\end{thm}

While in the unavoidability case one might arguably think that the very end of the regime is not that interesting since it only involves determining the behaviour of the second order term, we see here that this second order term becomes the main and only term for the case of universality. This regime in particular is related to previous work on universality for the family $\mathcal{E}(m)$ of graphs with exactly $m$ edges (and no isolated vertices). 
This was first considered by Babai, Chung, Erd\H{o}s, Graham  and Spencer \cite{babai1982graphs} in 1982. They show that there exist $\mathcal{E}(m)$-universal graphs with at most $\OO\left( \frac{m^2\log\log m}{\log m}\right)$ edges. This was later improved upon by Alon and Asodi \cite{alon2002sparse} who show the upper bound of $\OO\left(\frac{m^2}{\log^2m}\right),$ which matches the lower bound. Our bound on $g(2m,m)$ recovers the result of Alon and Asodi as any graph on $m$ edges can use at most $2m$ vertices, so any graph which is $\HH(2m,m)$-universal is also $\mathcal{E}(m)$-universal. On the other hand when $m>n/2$ their results 
can not be used to obtain a bound on $g(n,m)$ for the following reason. In order to bound $g(n,m)$ one needs to find an $\HH(n,m)$-universal graph on $n$ vertices. This can never come from an $\mathcal{E}(m)$-universal graph as any such graph needs to have at least $2m>n$ vertices, in order to contain a matching of $m$ edges. Furthermore, $\mathcal{E}(m)$-universal construction of Alon and Asodi uses more than $4m$ vertices. Since any graph in $\mathcal{E}(m)$ has at most $2m$ vertices, this affords them a lot of leeway during the embedding process. In contrast, in our problem we need to find spanning universal graphs, which requires a different embedding technique. Proving such spanning results is often a harder problem.
While neither our nor the universal graphs used by Alon and Asodi are explicit, our embedding technique allows us to work with a weaker and simpler property of random graphs, which might be helpful in answering their question about finding explicit universal graphs.

\par Another related notion of universality deals with the family $\mathcal{E}(n,d)$ of graphs on $n$ vertices with degree bounded by $d$. Alon and Capalbo \cite{alon2008optimal} show that $\Theta\left(n^{2-2/d}\right)$ is the least possible number of edges in an $\mathcal{E}(n,d)$-universal graph. The $\mathcal{E}(n,d)$-universal graph they construct has more than $(1+\varepsilon)n$ vertices and it is a seemingly hard open problem in the area to obtain such a graph on exactly $n$ vertices. The best result in this direction is due to the same authors \cite{alon2007sparse} where they exhibit such a graph with  $\OO\left(n^{2-2/d}\log^{4/d}n\right)$ edges. As already mentioned we need to overcome a similar difficulty in our setting since we require our universal graphs to be of the same order as the graphs we want to embed.

\textbf{Notation.} We will abbreviate $\HH(n,m)$-universal by  $(n,m)$-universal throughout the paper. Let $G=(V,E)$ be a graph and $U\subseteq V$. We denote with $G[U]$ the subgraph of $G$ induced by $U$. We denote with $\G(n,p)$ the standard Erd\H{o}s-Renyi random graph, i.e. the probability distribution on the set of all graphs on $n$ vertices where each graph $H$ has probability measure $p^{e(H)}(1-p)^{\binom{n}{2}-e(H)}$.
We say that $\G(n,p)$ satisfies a property with high probability (whp) if a sample from $\G(n,p)$ satisfies this property with probability tending to $1$ as $n$ tends to infinity. Let $f,g$ be functions from $\mathbb{N}$ to $\mathbb{R}^+$. Then $f=\OO(g)$ if there exists a constant $C$ such that $f(n)\leq Cg(n)$ for all $n\in\mathbb{N}$. Also, $f=\Omega(g)$ if $g=\OO(f)$. Furthermore, $f=\Theta(g)$ if $g=\OO(f)$ and $g=\Omega(f)$. If $\lim_{n\xrightarrow{} \infty}\frac{f(n)}{g(n)}=0$ then we write $f=o(g)$ and $g=\omega(f)$. With $\log n$ we denote the natural logarithm of $n$. We omit floors and ceilings whenever they are not essential.



\section[Section lower bounds]{Lower bounds for $g(n,m)$}\label{lower bounds}
We start our proof of Theorem \ref{thm:main2} with the lower bounds. 

\begin{lem}\label{lem:counting}
Let $G$ be an $(n,m)$-universal graph with $t$ edges. Then the following holds:
\begin{itemize}
    \item If $m>n\log n$ then $t\geq \binom{n}{2}\left(1-\frac{n\log n}{m}\right)$.
    \item If $m\leq n\log n$ then $t\geq \frac{1}{4}\binom{m/\log m}{2}$.
\end{itemize}
\end{lem}

\begin{proof}The idea behind the proof is quite simple, we will count how many non-isomorphic graphs there are on $n$ vertices and $m$ edges and we will count how many subgraphs with $m$ edges we can find in our a graph with $t$ edges. The existence of an $(n,m)$-universal graph $G$ on $t$ edges implies that the former count must be smaller than the latter (since all $n$-vertex $m$-edge graphs must appear as subgraphs of $G$); the desired bounds then follow after some simplifications, which will differ depending on the regime.

The number of non-isomorphic graphs on $n$ vertices and $m$ edges is at least
\[
    \binom{\binom{n}{2}}{m}\frac{1}{n!},
\]
here the first term counts the number of labelled graphs on $n$ vertices and $m$ edges, and we counted each graph at most $n!$ many times. 

Since $G$ contains all these graphs, and there are $\binom{t}{m}$ subgraphs of $G$ with $m$ edges, we get
\[
       \binom{\binom{n}{2}}{m}\frac{1}{n!}\leq \binom{t}{m}.
\]
By using the inequality $\binom{b}{c}\leq \binom{a}{c}(\frac{b}{a})^c$ which holds for all $a>b>c\geq 0$ we conclude 
\[
\frac{1}{n!}\leq \left( \frac{t}{\binom{n}{2}}\right)^m
\]
which gives 

\begin{equation} \label{NEMANJAAAAAAAAA}
t > \binom{n}{2}(n!)^{-1/m} \ge \binom{n}{2}n^{-n/m}.
\end{equation}
Since $n^{-n/m}\ge 1-\frac{n\log n}{m}$ the first claim follows.
For the case when $m< n\log n$ let $k=m/\log m<n.$ So $G$ is also $(k,m)$-universal and the bound \eqref{NEMANJAAAAAAAAA} implies 
$$ t > \binom{k}{2}k^{-k/m}=\binom{k}{2}\left(\frac{m}{\log m}\right)^{-1/\log m}>\binom{k}{2}m^{-1/\log m}=\binom{k}{2}e^{-1}.$$

\vspace{-0.7cm}
\end{proof}

\section[upper bounds]{Upper bounds for $g(n,m)$}
In order to give an upper bound on $g(n,m)$ one needs to exhibit an example of an $(n,m)$-universal graph. We distinguish two regimes, corresponding to the relations between $n$ and $m$ which give us different behaviours in \Cref{thm:main2}. Namely, in the first regime we will work with $n\log n<m<n^{3/2-\varepsilon}$ and in the second with $m \le n\log n$. In both regimes we have different examples but they share some common traits. We begin with a few results which will be useful for both regimes.

\par Let us state an immediate observation which we will use frequently, without further mention. 
\begin{obs}\label{deletion}
Let $G$ be a graph with $m$ edges. If $G'$ is the subgraph of $G$ obtained by deleting $k$ vertices of highest degree, then $\Delta(G')\leq 2m/k$. In fact, any vertex of $G'$ has degree at most $2m/k$ in $G$.
\end{obs}

The following lemma will allow us to pass from almost spanning universal graphs to spanning ones. Here, and in the rest of the paper when we say that we add full degree vertices to a graph we mean that we create a new graph which consists of an induced copy of the original graph and a number of new ``full degree'' vertices which are adjacent to every other vertex of the new graph.

\begin{lem}\label{lem:almost-spanning-to-spaning}
Let $n>k$ and suppose there is a $k$-vertex $(k-\floor{ \frac{k(n-k)}{2m+k}},m)$ universal graph $U$. By adding $n-k$ new full degree vertices to $U$ we obtain an $n$-vertex $(n,m)$-universal graph.
\end{lem}

\begin{proof}
Let $U$ be the universal graph given by the assumption. We construct $G$ by adding a set $F$ of $n-k$ new vertices to $U$ and joining each vertex of $F$ to all other vertices in $U \cup F$. We claim that $G$ is $(n,m)$-universal. To show this we consider an arbitrary $n$-vertex graph $H$ with $m$ edges. We prove that $H$ can be embedded into $G$. 

Before turning to the details let us explain the idea how we will achieve this. Let $\ell:=\floor{k(n-k)/(2m+k)}$ so that $U$ is $(k-\ell,m)$-universal. We will first find an independent set of $S$ of $\ell$ vertices in $H$ with the property that there are at least $k$ vertices of $H$ having no edges to $S$ (including vertices of $S$ themselves). We embed the remaining $n-k$ vertices into vertices of $G$ of full degree, we then embed the remaining $k-\ell$ vertices not in $S$ inside $U$, using the universality of $U$, finally we may embed the vertices in $S$ as we please to the remainder of $U$ since they send no edges to the part of $H$ embedded in $U$ and will be joined to every other vertex in $G$.

Let us now fill in the details. We pick a vertex $v_1$ in $H_0:=H$ of minimum degree and remove $v_1$ and all its neighbours from $H$ to obtain $H_1$. We repeat $\ell$ times, in step $i$ we pick a vertex $v_i$ in $H_{i-1}$ of minimum degree and create $H_i$ by removing $v_i$ and all its neighbours in $H_{i-1}$. By induction it is easy to see that after each step we are left with at least $k$ vertices in $H_{i}$. Indeed, if $H_{i-1}$ has at least $k$ vertices, then for $j \le i$ vertex $v_j$ has a degree of at most $2|E(H_{j-1})|/|H_{j-1}| \le 2m/k$ in $H_{j-1}$, since $v_j$ is a vertex of minimum degree in $H_{j-1}$ and $|H_{j-1}| \ge |H_{i-1}| \ge k$. So after deleting all such $v_j$'s and their neighbours, we have in total deleted at most $i\cdot (2m/k+1)\leq \ell(2m/k+1)\leq n-k$ vertices, so indeed we do have at least $k$ vertices in $H_i$ for each $i\in[\ell]$. 

In particular, we obtain a subgraph $H_\ell$ of $H$ of size at least $k$ and an independent set $S=\{v_1 \ldots, v_{\ell}\}$ of $\ell$ vertices each with no neighbours in $H_\ell$. We embed all vertices of $H \setminus (H_\ell \cup S)$ into $F$ (which we can since $|H \setminus H_\ell| \le n-k$) as well as as many vertices of $H_\ell$ as we can. If we are only left with vertices of $S$ we can embed them into $U$ (since they make an independent set). Otherwise we are left with a graph $H'$ such that $S \subseteq H'\subseteq H_\ell \cup S$. Note that $H'$ has $|G \setminus F|=k$ vertices, $\ell$ of which are independent (depicted by the first and last part of $H$ in \cref{fig:AlmostSpanning}). Since $U$ was $(k-\ell,m)$-universal we can embed $H'\setminus S$ into $U$ and place $S$ into unused vertices. This is permissible since vertices in $S$ were isolated in $H'$. Furthermore, since all other vertices got embedded into $F$ (whose vertices have full degree in $G$), we have found our embedding of $H$.
\end{proof}

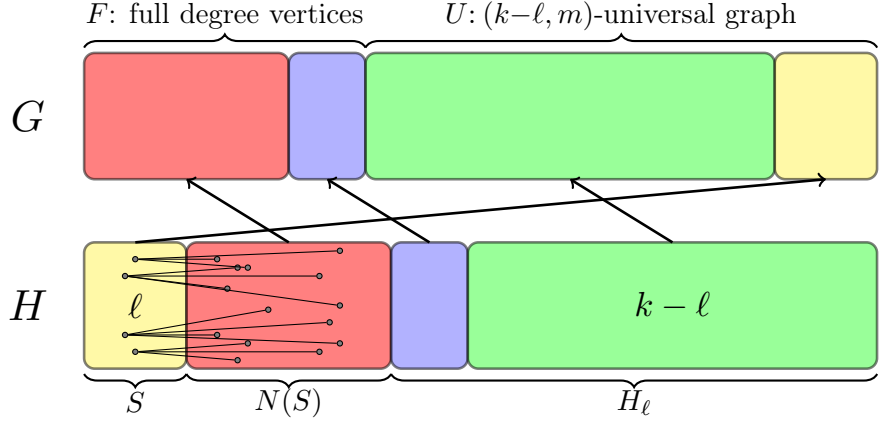
\begin{figure}
\centering
\begin{tikzpicture}[scale = 0.8,yscale=0.7, xscale=1.7]
\defPt {1}{}{a1}

\defPt {0}{0}{a1};
\defPt {2}{3}{a2};
\defPt {2.75}{0}{a3};
\defPt {6.75}{3}{a4};
\defPt {7.75}{0}{a5};

\draw[rounded corners, opacity = 0.5, fill=red!100,line width =1 pt] (a1) rectangle (a2);
\draw[rounded corners, opacity = 0.5, fill=blue!60,line width =1 pt] (a2) rectangle (a3);
\draw[rounded corners, opacity = 0.5, fill=green!80,line width =1 pt] (a3) rectangle (a4);
\draw[rounded corners, opacity = 0.5, fill=yellow!80,line width =1 pt] (a4) rectangle (a5);

\defPtm{($0.5*(a1)+0.5*(a2)$)}{b1}
\defPtm{($0.5*(a2)+0.5*(a3)$)}{b2}
\defPtm{($0.5*(a3)+0.5*(a4)$)}{b3}
\defPtm{($0.5*(a4)+0.5*(a5)$)}{b4}

\defPt {0}{-1.5}{c1}
\defPt {1}{-4.5}{c2}
\defPt {3}{-1.5}{c3}
\defPt {3.75}{-4.5}{c4}
\defPt {7.75}{-1.5}{c5}

\draw[rounded corners, opacity = 0.5, fill=yellow!80,line width =1 pt] (c1) rectangle (c2);
\draw[rounded corners, opacity = 0.5, fill=red!100,line width =1 pt] (c2) rectangle (c3);
\draw[rounded corners, opacity = 0.5, fill=blue!60,line width =1 pt] (c3) rectangle (c4);
\draw[rounded corners, opacity = 0.5, fill=green!80,line width =1 pt] (c4) rectangle (c5);

\defPtm{($0.5*(c1)+0.5*(c2)$)}{b1}
\defPtm{($0.5*(c2)+0.5*(c3)$)}{b2}
\defPtm{($0.5*(c3)+0.5*(c4)$)}{b3}
\defPtm{($0.5*(c4)+0.5*(c5)$)}{b4}

\node[scale=1.2] (b) at (b1)  {$\ell$};
\node[scale=1] (b) at (b2)  {};
\node[scale=1] (b) at (b3)  {};
\node[scale=1.2] (b) at (b4) {$k-\ell$};

\defPt{0}{1.5}{s}

\defPtm{($0.5*(c1)+0.5*(c2)+(s)$)}{p1}
\defPtm{($0.5*(c2)+0.5*(c3)+(s)$)}{p2}
\defPtm{($0.5*(c3)+0.5*(c4)+(s)$)}{p3}
\defPtm{($0.5*(c4)+0.5*(c5)+(s)$)}{p4}

\defPt{0}{-1.5}{s}
\defPtm{($0.5*(a1)+0.5*(a2)+(s)$)}{q1}
\defPtm{($0.5*(a2)+0.5*(a3)+(s)$)}{q2}
\defPtm{($0.5*(a3)+0.5*(a4)+(s)$)}{q3}
\defPtm{($0.5*(a4)+0.5*(a5)+(s)$)}{q4}


\defPt{0}{3}{s}
\defPtm{($(a1)+(s)$)}{V_11}
\defPtm{($(a3)+(s)$)}{V_12}
\defPtm{($(a3)+(s)$)}{V_21}
\defPtm{($(a5)+(s)$)}{V_22}

  \draw [thick,decorate,decoration={brace,  amplitude=5pt,raise=2pt},yshift=0pt]($(V_11)$) -- ($(V_12)$) node [black,midway,yshift=0.5cm] {$F$: full degree vertices};
   \draw [thick,decorate,decoration={brace,  amplitude=5pt,raise=2pt},yshift=0pt]($(V_21)$) -- ($(V_22)$) node [black,midway,yshift=0.5cm] {$U{:}\:(k{-}\ell,m)\text{-universal graph}$};

  \defPt{0}{-3}{s}
\defPtm{($(c5)+(s)$)}{R1}
\defPtm{($(c3)+(s)$)}{R2}

\defPtm{($(c1)+(s)$)}{S1}
\defPtm{($(c2)$)}{S2}

\defPtm{($(c2)$)}{N1}
\defPtm{($(c3)+(s)$)}{N2}

   \draw [thick,decorate,decoration={brace,  amplitude=5pt,raise=2pt},yshift=0pt]($(R1)$) -- ($(R2)$) node [black,midway,yshift=-0.45cm] {$H_\ell$};
   
      \draw [thick,decorate,decoration={brace,  amplitude=5pt,raise=2pt},yshift=0pt]($(S2)$) -- ($(S1)$) node [black,midway,yshift=-0.45cm] {$S$};
      
       \draw [ thick,decorate,decoration={brace,  amplitude=5pt,raise=2pt},yshift=0pt]($(N2)$) -- ($(N1)$) node [black,midway,yshift=-0.45cm] {$N(S)$};
      
  
   \draw[->,line width= 1 pt] (p1) -- (q4);
   \draw[->,line width= 1 pt] (p2) -- (q1);
   \draw[->,line width= 1 pt] (p3) -- (q2);
   \draw[->,line width= 1 pt] (p4) -- (q3);
   
  \node[scale=1.5] (G) at(-0.55,1.5) {$G$};
  \node[scale=1.5] (H) at(-0.55,-3) {$H$};

   \tikzstyle{every node}=[circle, draw, fill=black!50,
                        inner sep=0pt, minimum width=2pt]
   \node[] (t1) at (0.5,-4.1){};
   \node[] (n1) at (1.5,-4.3){};
   \node[] (n2) at (2.3,-4.1){};
   \node[] (n3) at (1.6,-3.9){};
   
   \node[] (t2) at (0.4,-3.7){};
   \node[] (n4) at (2.4,-3.4){};
   \node[] (n5) at (1.3,-3.7){};
   \node[] (n6) at (2.5,-3.9){};
   \node[] (n7) at (1.8,-3.1){};

   \draw[color=black] (t1) to (n1);
   \draw[color=black] (t1) to (n2);
   \draw[color=black] (t1) to (n3);
   
   \draw[color=black] (t2) to (n4);
   \draw[color=black] (t2) to (n5);
   \draw[color=black] (t2) to (n6);
   \draw[color=black] (t2) to (n7);

   
   \node[] (t1) at (0.5,-1.9){};
   \node[] (n1) at (2.5,-1.7){};
   \node[] (n2) at (1.3,-1.9){};
   \node[] (n3) at (1.6,-2.1){};
   
   \node[] (t2) at (0.4,-2.3){};
   \node[] (n4) at (1.4,-2.6){};
   \node[] (n5) at (2.3,-2.3){};
   \node[] (n6) at (1.5,-2.1){};
    \node[] (n7) at (2.5,-3){};

   \draw[color=black] (t1) to (n1);
   \draw[color=black] (t1) to (n2);
   \draw[color=black] (t1) to (n3);
   
   \draw[color=black] (t2) to (n4);
   \draw[color=black] (t2) to (n5);
   \draw[color=black] (t2) to (n6);
   \draw[color=black] (t2) to (n7);

\end{tikzpicture}

    \caption{Illustration of the final state of the embedding process when passing from almost spanning universality to spanning universality.} \label{fig:AlmostSpanning}
\end{figure}

This lemma shows that if we can establish a sufficiently strong version of almost spanning universality it only costs us a very few extra vertices of full-degree in order to obtain spanning universality. On the other hand since we want universality for $n$-vertex graphs with $m$ edges, we know that provided $m\ge n-1$ the host graph must contain some vertices of full degree, as the graphs we want to embed might also have some vertices of full degree. We note that this simple lemma is one of the main ingredients which allow us to find spanning universal graphs, which is usually a very hard problem. In addition it allows us to greatly simplify certain technical parts of our arguments. 

We will construct a part of our universal graph $G$ to satisfy a property of the following type which will allow us to embed any graph on $n$ vertices and $m$ edges into $G$.

 \begin{defn*}\label{defn:pseudorandom}
 A graph is said to have the $(r,s,t)$-domination property if for any set $R$ of size $r$ and any $t$ disjoint sets $S_1,\ldots, S_t$ of size $s$ which are disjoint from $R$, there is a vertex $v$ in $R$ and a set $S_i$ such that $v$ is a common neighbour of all vertices in $S_i$.
 \end{defn*}

A part of our universal graphs which will satisfy the above property is going to be provided by a random graph. Towards this end the following lemma establishes a condition on the parameters $r,s,t$ which implies that the random graph $\G(n,p)$ has the $(r,s,t)$-domination property.

\begin{lem}\label{lem:random-is-pseudorandom}
 The random graph $\G(n,p)$ has the $(r,s,t)$-domination property whp provided
 $$3 \log n   \le p^s\cdot  \min\left(\frac{r}{s},t\right).$$
\end{lem}

\begin{proof}
The probability that a fixed vertex is a common neighbour of every vertex in a fixed set $S$ of size $s$ is $p^s$. By independence, the probability that a fixed set $R$ of size $r$ and $t$ fixed disjoint sets $S_1,\ldots, S_t$ of size $s$ are such that no vertex in $R$ is a common neighbour of every vertex in some $S_i$ is $(1-p^s)^{tr}$. So, by a union bound, the probability that some choice of sets $R,S_1,\ldots,S_t$ fails the desired condition is at most
$$\binom{n}{r}\binom{n}{s}^t(1-p^s)^{tr} \le n^{r + ts} e^{-p^str}=
 n^{r + ts -p^str/\log n} \le  n^{-r} = o(1). $$
\end{proof}

\subsection{The first regime}\label{The first regime}
Let us now turn to the first regime, so when $n\log n<m<n^{3/2-\varepsilon}.$ As part of our construction of an $(n,m)$-universal graph with $n$ vertices we first find a smaller almost spanning universal graph, in particular a graph on $k<n$ vertices which is $(k-r,m)$-universal for appropriately chosen parameters $k$ and $r$. We then finish the construction by applying \Cref{lem:almost-spanning-to-spaning}.

Let $\eps>0$ and $k,m$ be integers such that $k \log k  < m < 3k^{3/2-\eps}$ and $k$ is sufficiently large depending only on $\eps$.
We let $G$ be a $k$-vertex graph consisting of 
\begin{itemize}
    \item a set $F$ of $k/2$ vertices joined to every vertex and
    \item a set $V$ of $k/2$ vertices inducing a graph $G'$ with the $(r,s,t)$-domination property missing at least $\frac{\eps k^3 \log k}{2^8m}$ edges
\end{itemize}  where $t=\frac{k^3}{2^{9}m^2}$, $s=\frac{8m}{k}$ and  $r=\frac{k^2}{4m}$. In order to ensure we can find such a graph $G'=G[V]$ we let $p=1-\frac{\eps k \log k}{16m}>\frac12$. Then we have $$p^s=(1-(1-p))^s \ge e^{-2(1-p)s}=k^{-\eps},$$ where we used that $1-x \geq e^{-2x}$,  provided $0\le x<1/2$. In addition, since $m<3k^{3/2-\eps}$, we also have
$$\min\left(\frac{r}{s},t\right)=t=k^3/(2^{9}m^2)\ge k^{2\eps}/(3^2 2^{9})> 3 k^{\eps}\log k,$$ as $k$ is large enough compared to $\eps$. So, by Lemma \ref{lem:random-is-pseudorandom} $\G(k/2,p)$ has the $(r,s,t)$-domination property whp. Notice further that $\G(k/2,p)$ is also missing at least $(1-p)k^2/16=\frac{\eps k^3 \log k}{2^8m}$ edges whp. So in particular it provides us with our $G'$. We now show $G$ is universal for almost spanning graphs with $m$ edges.

\begin{claim*}
$G$ is $\left(k-r,m\right)$-universal. 
 \end{claim*}
\vspace{-0.5cm} 
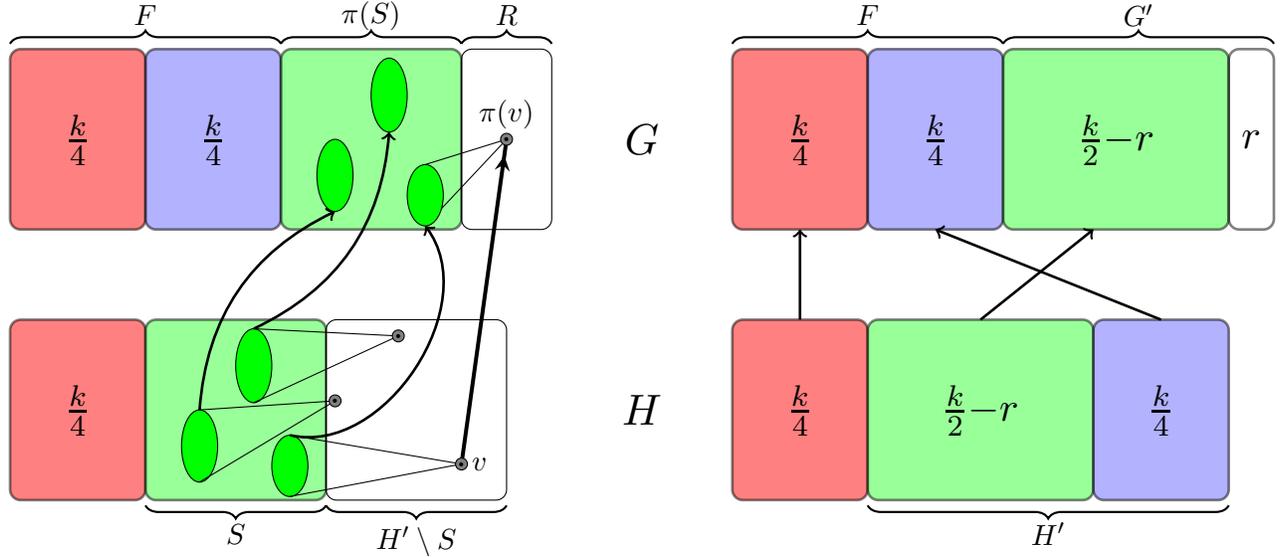
\begin{figure}
\centering
\begin{tikzpicture}[scale = 2, xscale=1.2, yscale=1.2]

\draw[rounded corners, opacity = 0.5, fill=red!100,line width =1 pt] (-1.5,0) rectangle (-0.75,1);
\draw[rounded corners, opacity = 0.5, fill=blue!60,line width =1 pt] (-0.75,0) rectangle (-0,1);
\draw[rounded corners, opacity = 0.5, fill=green!80,line width =1 pt] (0,0) rectangle (1.25,1);
\draw[rounded corners, opacity = 0.5, line width =1 pt] (1.25,0) rectangle (1.5,1);
\node[scale=1.4] (a) at  (-1.125,0.5)  {$\frac{k}{4}$};
\node[scale=1.4] (b) at  (-1.125+0.75,0.5)  {$\frac{k}{4}$};
\node[scale=1.4] (c) at  (1.25/2,0.5)  {$\frac{k}{2}{-}r$};
\node[scale=1.4] (d) at (1.5-0.125,0.5){$r$};

\draw[rounded corners, opacity = 0.5, fill=red!100,line width =1 pt] (-1.5,-1.5) rectangle (-0.75,-0.5);
\draw[rounded corners, opacity = 0.5, fill=green!80,line width =1 pt] (-0.75,-1.5) rectangle (-0.75+1.25,-0.5);
\draw[rounded corners, opacity = 0.5, fill=blue!60,line width =1 pt] (0.5,-1.5) rectangle (1.25,-0.5);
\node[scale=1.4] (a) at  (-1.125,-1)  {$\frac{k}{4}$};
\node[scale=1.4] (b) at  (-0.75+1.25/2,-1)  {$\frac{k}{2}{-}r$};
\node[scale=1.4] (c) at  (-0.75+1.25+0.375,-1)  {$\frac{k}{4}$};

\defPt{-1.5+0.375}{-0.5}{r1};
\defPt{-1.5+0.375}{0}{r2};

\defPt{1.5-0.25-0.375}{-0.5}{b1};
\defPt{-0.75+0.375}{0}{b2};

\defPt{-0.75+1.25/2}{-0.5}{g1};
\defPt{1.25-0.75}{0}{g2};

\defPt{-1.5}{1}{F1};
\defPt{0}{1}{F2};

\defPt{0}{1}{G'1};
\defPt{1.5}{1}{G'2};

\defPt{1.25}{-1.5}{H'1};
\defPt{-0.75}{-1.5}{H'2};
 \draw [thick,decorate,decoration={brace,  amplitude=5pt,raise=2pt},yshift=0pt]($(F1)$) -- ($(F2)$) node [black,midway,yshift=0.45cm] {$F$};
 
  \draw [thick,decorate,decoration={brace,  amplitude=5pt,raise=2pt},yshift=0pt]($(G'1)$) -- ($(G'2)$) node [black,midway,yshift=0.45cm] {$G'$};
  
  \draw [thick,decorate,decoration={brace,  amplitude=5pt,raise=2pt},yshift=0pt]($(H'1)$) -- ($(H'2)$) node [black,midway,yshift=-0.45cm] {$H'$};
  \node[scale=1.5] (G) at(-2,0.5) {$G$};
  \node[scale=1.5] (H) at(-2,-1) {$H$};
  
  \draw[->,line width= 1 pt] (r1) -- (r2);
  \draw[->,line width= 1 pt] (b1) -- (b2);
  \draw[->,line width= 1 pt] (g1) -- (g2);
  
\defPt{-1.5}{0}{a1}
\defPt{-0.75}{1}{a2}
\defPt{0}{0}{a3}
\defPt{1}{1}{a4}
\defPt{1.5}{0}{a5}
\defPt{-4}{0}{s}
\foreach \i in {1,...,5}
{
\defPtm{($(a\i)+(s)$)}{a\i};
};

\draw[rounded corners, opacity = 0.5, fill=red!100,line width =1 pt] (a1) rectangle (a2);
\draw[rounded corners, opacity = 0.5, fill=blue!60,line width =1 pt] (a2) rectangle (a3);
\draw[rounded corners, opacity = 0.5, fill=green!80,line width =1 pt](a3) rectangle (a4);
\draw[rounded corners] (a4) rectangle (a5);

\defPt{-1.5}{-1.5}{b1}
\defPt{-0.75}{-0.5}{b2}
\defPt{0.25}{-1.5}{b3}
\defPt{1.25}{-0.5}{b4}
\foreach \i in {1,...,4}
{
\defPtm{($(b\i)+(s)$)}{b\i};
};

\draw[rounded corners, opacity = 0.5, fill=red!100,line width =1 pt] (b1) rectangle (b2);
\draw[rounded corners, opacity = 0.5, fill=green!80,line width =1 pt] (b2) rectangle (b3);
\draw[rounded corners] (b3) rectangle (b4);

\defPtm{($0.5*(b2)+0.5*(b3)$)}{e1}
\defPt{0.1}{0}{r}
\defPt{0}{-0.1}{d}
\defPt{0}{0.1}{u}
\defPt{-0.1}{0}{l}

\defPtm{($(e1)+2*(l)+4*(d)$)}{t1}
\defPtm{($(e1)+2*(l)+0*(u)$)}{t2}

\defPtm{($(e1)+1*(r)+4.5*(u)$)}{t3}
\defPtm{($(e1)+1*(r)+0.4*(u)$)}{t4}

\defPtm{($(e1)+3*(r)-4.8*(u)$)}{t5}
\defPtm{($(e1)+3*(r)-1.4*(u)$)}{t6}

\fitellipsis{t1}{t2}
\fitellipsis{t3}{t4}
\fitellipsis{t5}{t6}

\defPt{0}{1}{s}
\defPtm{($(a3)-(b2)+(s)$)}{trans}
\foreach \i in {1,...,6}
{
\defPtm{($(t\i)+(trans)$)}{g\i}
}

  \defPt{-3.35}{-0.59}{p1}
    \defPt{-3.7}{-0.95}{p2}
      \defPt{-3}{-1.3}{p3}

     \draw[color=black] (p2) to (t1);
     \draw[color=black] (p2) to (t2);
     
       \draw[color=black] (p1) to (t3);
     \draw[color=black] (p1) to (t4);
     
       \draw[color=black] (p3) to (t5);
     \draw[color=black] (p3) to (t6);

     \draw[->,line width= 1 pt] (t2) to [bend left=30](g1);
     \draw[->,line width= 1 pt] (t3) to [bend right=30](g4);
    \draw[->,line width= 1 pt] (t6) to [bend right=70](g5);
  
  \defPt{0}{-1}{s}
   \defPtm{($(b2)+(s)$)}{B2}
   \draw [thick,decorate,decoration={brace,  amplitude=5pt,raise=2pt},yshift=0pt]($(b3)$) -- ($(B2)$) node [black,midway,yshift=-0.45cm] {$S$};
   
     \defPt{0}{1}{s}
   \defPtm{($(a3)+(s)$)}{A3}
   \draw [thick,decorate,decoration={brace,  amplitude=5pt,raise=2pt},yshift=0pt]($(A3)$) -- ($(a4)$) node [black,midway,yshift=0.45cm] {$\pi(S)$};
   
      \defPtm{($(b4)-(s)$)}{B4}
      \draw [thick,decorate,decoration={brace,  amplitude=5pt,raise=2pt},yshift=0pt]($(B4)$) -- ($(b3)$) node [black,midway,yshift=-0.55cm] {$H'\setminus S$};
      
      \defPtm{($(a1)+(s)$)}{A1}
      \defPtm{($(a3)+(s)$)}{A3}
      \draw [thick,decorate,decoration={brace,  amplitude=5pt,raise=2pt},yshift=0pt]($(A1)$) -- ($(A3)$) node [black,midway,yshift=+0.45cm] {$F$};

      \defPtm{($(a5)+(s)$)}{A5}
      \draw [thick,decorate,decoration={brace,  amplitude=5pt,raise=2pt},yshift=0pt]($(a4)$) -- ($(A5)$) node [black,midway,yshift=+0.45cm] {$R$};

     \defPtm{($0.5*(a4)+0.5*(a5)$)}{p}
     \draw[diredge, line width= 1.5 pt] (p3) to (p);
    \node[scale=1] (c) at  ([yshift=0.14cm]p)  {$\pi(v)$};
    \node[scale=1] (c) at  ([xshift=0.1cm]p3)  {$v$};
   
     \draw[color=black] (p) to (g5);
     \draw[color=black] (p) to (g6);

     \defPtm{($0.5*(a1)+0.5*(a2)$)}{k1}
     \defPtm{($0.5*(a2)+0.5*(a3)$)}{k2}
     \defPtm{($0.5*(b1)+0.5*(b2)$)}{k3}
     
      \node[scale=1.4] (c) at  (k1)  {$\frac{k}{4}$};
     \node[scale=1.4] (c) at  (k2)  {$\frac{k}{4}$};
     \node[scale=1.4] (c) at  (k3)  {$\frac{k}{4}$};

     \tikzstyle{every node}=[circle, draw, fill=black!50,
                        inner sep=0pt, minimum width=2pt]
 \node[scale=1.4] (c) at  (p1)  {.};
 \node[scale=1.4] (c) at  (p2)  {.};
 \node[scale=1.4] (c) at  (p3)  {.}; 
 \node[scale=1.4] (c) at  (p)   {.};

\fitellipsis{g1}{g2}
\fitellipsis{g3}{g4}
\fitellipsis{g5}{g6}
  
\end{tikzpicture}

\caption{The picture on the left illustrates the embedding process while the one on the right illustrates the final state of the embedding.} \label{fig:R1}
\end{figure}

\begin{proof}

 Let $H$ be a graph on $k-r$ vertices and $m$ edges. Our goal is to embed $H$ into $G$. 
 Let us first delete $k/4$ vertices of largest degrees from $H$ to obtain $H'\subseteq H$ with $\Delta(H') \le 8m/k$ and $|H'| = 3k/4-r$. The deleted $k/4$ vertices we embed into $F$, leaving us with $k/4$ remaining vertices in $F$. 

We continue by embedding vertices from $H'$ one by one into $G'$, making sure that whenever two embedded vertices make an edge in $H'$, their images also make an edge in $G'$. Let $S$ be the set of already embedded vertices from $H'$, i.e.\ assume that we have already found a distinct vertex $\pi(v)\in V(G')$ for every $v \in S \subseteq V(H')$. If $|S| \ge k/2-r$ we stop, so let us assume that $|S| < k/2-r$, which implies that we have at least $k/4$ vertices in $H'\setminus S$. For each $v\in H'\setminus S$, let $S_v=\pi(S \cap N_{H'}(v))$, so the set of vertices of $G'$ which were already assigned to a neighbour of $v$ (see \cref{fig:R1}). We know that $|S_v| \le 8m/k=s$ since $\Delta(H') \le 8m/k$. Furthermore, the same max degree condition tells us that each $S_v$ can intersect at most $(8m/k)^2$ other such sets (since $S_v$ has size at most $8m/k$ and every vertex of $S_v$ can belong to at most $8m/k$ other $S_u$'s). This implies the existence of a family of at least $\frac{|H'\setminus S|}{(8m/k)^2+1}\geq \frac{k^3}{2^{9}m^2}=t$ disjoint sets $S_v$. We can find such a family by greedily choosing the $S_v$'s.
 
On the other hand we have a set $R$ of at least $|G'|-|S|\geq r$ yet unassigned vertices of $G'$. Since $G'$ is $(r,s,t)$-dominating we know that there is a vertex $u$ in $R$ and one of the sets $S_v$ such that $u$ is a common neighbour of every vertex in $S_v$. So we may embed $\pi(v)=u$ and repeat.
 
When we stop (i.e. when $|S|\geq k/2-r$), we will have embedded at least $k/2-r$ vertices of $H'$. We now embed the remaining $k/4$ vertices of $H'$ to the remaining $k/4$ vertices of $F$ to obtain the desired copy of $H'$ in $G'.$ 
\end{proof}

\begin{thm}\label{thm:upper-bound-big}
Let $\eps>0$ and $n \log n < m < n^{3/2-\eps}$. There is an $n$-vertex graph missing at least $\frac{\eps n^3 \log n}{2^{12}m}$ edges which is $(n,m)$-universal, provided $n$ is large enough depending only on $\eps$.
\end{thm}
\begin{proof}
Let $k=n/2$, so $k \log k < m < 3k^{3/2-\eps}$. By the previous claim $G$ is a $(k-r,m)$-universal graph which misses at least $\frac{\eps k^3 \log k}{2^8m}\ge \frac{\eps n^3 \log n}{2^{12}m}$ edges. Since 
$r= \frac{k^2}{4m} \le \floor{\frac{k(n-k)}{2m+k}}$ Lemma \ref{lem:almost-spanning-to-spaning} implies that adding $n-k=n/2$ vertices of full degree to $G$ gives us the desired graph.
\end{proof}

Note that our final example for above theorem consisted of $3n/4$ vertices of full degree and $n/4$ vertices spanning a graph satisfying certain properties which hold for the random graph $\G(n/4,p)$ whp. Upon taking complements, what we have shown is that whp the random graph $\G(n/4,q)$, where $q=1-p,$ has Tur\'an number on $n$ vertices of at most $\binom{n}{2}-m=\binom{n}{2}-\frac{\eps n \log \frac{n}{2}}{32q}=(1-\Theta(\frac{\log n}{nq}))\binom{n}{2}$, where we may choose an arbitrary $q$ such that $1/2>q>n^{-1/2+\eps}$. Since in this regime $\chi(\G(n,q))= \Theta (nq/\log n)$ this is approximately what we would obtain by applying Erd\H{o}s-Stone-Simonovits theorem, despite the fact that our graph is of linear order while the usual Erd\H{o}s-Stone-Simonovits applies only to small forbidden graphs. We postpone further discussion to concluding remarks.

\subsection{The second regime}

In this subsection we deal with the case when $n/2<m<n\log n$. In this regime our goal is to construct an $(n,m)$-universal graph on $\Theta\left(\frac{m^2}{\log^2m}\right)$ edges to show the desired upper bound.

A part of our construction will again use a graph satisfying an appropriate domination property with few edges, similarly as in the first regime. We however require slightly different relation between parameters. The next lemma shows that the random graph still provides us with such a graph.

\begin{lem}\label{Gnp}
 $\G(n,p)$ whp has the $(n^{\frac34},\frac{\log n}{2\log\frac{1}{p}},n^{\frac34})$-domination property for $p<\frac12$.
\end{lem}{}

\begin{proof}
 Using \Cref{lem:random-is-pseudorandom} and setting $r=n^{\frac34}$, $s=\frac{\log n}{2\log\frac{1}{p}}$ and $t=n^{\frac34}$ , it is enough to show that $3 \log n   \le p^s\cdot  \min\left(\frac{r}{s},t\right)$, which is equivalent to
 $3\log n\leq \frac{1}{\sqrt{n}}\cdot \frac{n^{\frac34}}{s}$
 and since $s\leq \log n$, we are done.
\end{proof}

To upper bound $g(n,m)$ we will give a recursive construction of an $(n,m)$-universal graph. It will not be hard to ``extract'' the constructed graph later, but we use the recursive definition as it provides us with a convenient way of controlling the bounds. 

\begin{lem}\label{lem:recursive step}
For $n/2 \le m \le \frac{n \log n}{2^{10}}$ we have $$g(n,m) \le \frac{32m^3}{n \log ^3 n}+g\left(n' , m\right)$$ where $n'=32m \cdot \log(\frac{n\log n}{m})/\log n$.
\end{lem}
\begin{proof}
Note that $n'$ is increasing in $m$ for $m \le \frac{n \log n}{2^{10}}$ so
\begin{equation}\label{n'<n}
    n'\leq 32 \frac{n\log n}{2^{10}}\cdot\frac{\log(2^{10})}{\log n}\leq\frac{2}{3}n.
\end{equation}
 Let $k=n-\frac{n}{\log^3 n}$. Our initial goal is to find a $k$-vertex, $(k-n^{\frac45},m)$-universal graph $G$ with at most $g(n',m)+\frac{17m^3}{n\log^3 n}$ edges. After this is done, we will finish the proof by applying \Cref{lem:almost-spanning-to-spaning}.
\par 
We construct $G$ as follows. Let $p=\frac{m^3}{n^3\log^3n}<\frac12$. Let the vertex set $V$ of $G$ be the union of three disjoint sets $V=V_1\cup V_2\cup V_3$, of sizes $|V_1|=\frac{n}{\log^3n}+n^{\frac45}$, $|V_2|=n'$ and $|V_3|=k-|V_1|-|V_2|.$ 

We construct the edge set in three steps as follows:
\begin{itemize}
    \item  Let $V$ induce a graph with the $(n^{\frac34},\frac{\log n}{2\log\frac{1}{p}},n^{\frac34})$-domination property which has at most $n^2p$ edges. 
    \item Make each vertex in $V_1$ adjacent to all other vertices.
    \item Finally, add at most $g(n',m)$ edges within $V_2$ to make $G[V_2]$ an $(n',m)$-universal graph.
\end{itemize}
The graph from the first step exists due to \Cref{Gnp}, as $\G(n,p)$ has both the $(n^{\frac34},\frac{\log n}{2\log\frac{1}{p}},n^{\frac34})$-domination property and at most $n^2p$ edges whp, and since both these properties are hereditary we may take any subgraph on $k<n$ vertices for our graph. Observe that $G$ has at most $$n^2p+n|V_1|+g(n',m) \le n^2p+\frac{2n^2}{\log^3n}+g(n',m)\leq \frac{17m^3}{n\log^3n}+g(n',m)$$ edges, as $n/2<m$.

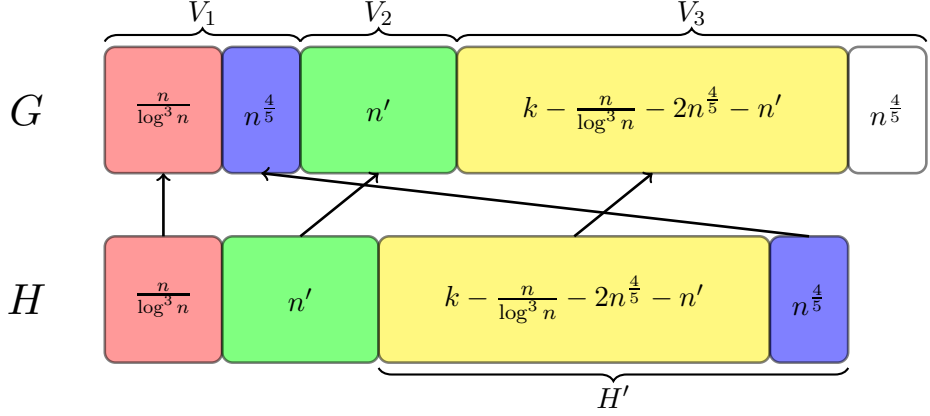
\begin{figure}
\centering
\begin{tikzpicture}[scale = 0.8, xscale=1.3, yscale=0.7]
\defPt {1}{}{a1}

\defPt {-5.5}{0}{a1};
\defPt {-4}{3}{a2};
\defPt {-3}{0}{a3};
\defPt {-1}{3}{a4};
\defPt {4}{0}{a5};
\defPt {5}{3}{a6};

\draw[rounded corners, opacity = 0.5, fill=red!80,line width =1 pt] (a1) rectangle (a2);
\draw[rounded corners, opacity = 0.5, fill=blue!100,line width =1 pt] (a2) rectangle (a3);
\draw[rounded corners, opacity = 0.5, fill=green!100,line width =1 pt] (a3) rectangle (a4);
\draw[rounded corners, opacity = 0.5, fill=yellow!100,line width =1 pt] (a4) rectangle (a5);
\draw[rounded corners, opacity = 0.5, line width =1 pt] (a5) rectangle (a6);

\defPtm{($0.5*(a1)+0.5*(a2)$)}{b1}
\defPtm{($0.5*(a2)+0.5*(a3)$)}{b2}
\defPtm{($0.5*(a3)+0.5*(a4)$)}{b3}
\defPtm{($0.5*(a4)+0.5*(a5)$)}{b4}
\defPtm{($0.5*(a5)+0.5*(a6)$)}{b5}

\node[scale=0.9] (b) at (b1)  {$\frac{n}{\log^3n}$};
\node[scale=1] (b) at (b2)  {$n^{\frac45}$};
\node[scale=1] (b) at (b3)  {$n'$};
\node[scale=1] (b) at (b4)  {$k-\frac{n}{\log^3n}-2n^{\frac45}-n'$};
\node[scale=1] (b) at (b5)  {$n^{\frac45}$};

\defPt {-5.5}{-1.5}{c1}
\defPt {-4}{-4.5}{c2}
\defPt {-2}{-1.5}{c3}
\defPt {3}{-4.5}{c4}
\defPt {4}{-1.5}{c5}

\draw[rounded corners, opacity = 0.5, fill=red!80, line width =1 pt] (c1) rectangle (c2);
\draw[rounded corners, opacity = 0.5, fill=green!100, line width =1 pt] (c2) rectangle (c3);
\draw[rounded corners, opacity = 0.5, fill=yellow!100, line width =1 pt] (c3) rectangle (c4);
\draw[rounded corners, opacity = 0.5, fill=blue!100, line width =1 pt] (c4) rectangle (c5);

\defPtm{($0.5*(c1)+0.5*(c2)$)}{b1}
\defPtm{($0.5*(c2)+0.5*(c3)$)}{b2}
\defPtm{($0.5*(c3)+0.5*(c4)$)}{b3}
\defPtm{($0.5*(c4)+0.5*(c5)$)}{b4}

\node[scale=0.9] (b) at (b1)  {$\frac{n}{\log^3n}$};
\node[scale=1] (b) at (b2)  {$n'$};
\node[scale=1] (b) at (b3)  {$k-\frac{n}{\log^3n}-2n^{\frac45}-n'$};
\node[scale=1] (b) at (b4) {$n^{\frac45}$};

\defPt{0}{1.5}{s}

\defPtm{($0.5*(c1)+0.5*(c2)+(s)$)}{p1}
\defPtm{($0.5*(c2)+0.5*(c3)+(s)$)}{p2}
\defPtm{($0.5*(c3)+0.5*(c4)+(s)$)}{p3}
\defPtm{($0.5*(c4)+0.5*(c5)+(s)$)}{p4}

\defPt{0}{-1.5}{s}
\defPtm{($0.5*(a1)+0.5*(a2)+(s)$)}{q1}
\defPtm{($0.5*(a2)+0.5*(a3)+(s)$)}{q2}
\defPtm{($0.5*(a3)+0.5*(a4)+(s)$)}{q3}
\defPtm{($0.5*(a4)+0.5*(a5)+(s)$)}{q4}





\defPt{0}{3}{s}
\defPtm{($(a1)+(s)$)}{V_11}
\defPtm{($(a3)+(s)$)}{V_12}
\defPtm{($(a3)+(s)$)}{V_21}
\defPtm{($(a4)$)}{V_22}
\defPtm{($(a4)$)}{V_31}
\defPtm{($(a6)$)}{V_32}

  \draw [thick,decorate,decoration={brace,  amplitude=5pt,raise=2pt},yshift=0pt]($(V_11)$) -- ($(V_12)$) node [black,midway,yshift=0.45cm] {$V_1$};
   \draw [thick,decorate,decoration={brace,  amplitude=5pt,raise=2pt},yshift=0pt]($(V_21)$) -- ($(V_22)$) node [black,midway,yshift=0.45cm] {$V_2$};
    \draw [thick,decorate,decoration={brace,  amplitude=5pt,raise=2pt},yshift=0pt]($(V_31)$) -- ($(V_32)$) node [black,midway,yshift=0.45cm] {$V_3$};

  \defPt{0}{-3}{s}
\defPtm{($(c5)+(s)$)}{R1}
\defPtm{($(c3)+(s)$)}{R2}
   \draw [thick,decorate,decoration={brace,  amplitude=5pt,raise=2pt},yshift=0pt]($(R1)$) -- ($(R2)$) node [black,midway,yshift=-0.45cm] {$H'$};
  
   \draw[->,line width= 1 pt] (p1) -- (q1);
   \draw[->,line width= 1 pt] (p2) -- (q3);
   \draw[->,line width= 1 pt] (p3) -- (q4);
   \draw[->,line width= 1 pt] (p4) -- (q2);
   
   \node[scale=1.5] (G) at(-6.5,1.5) {$G$};
  \node[scale=1.5] (H) at(-6.5,-3) {$H$};
   
 \end{tikzpicture}

\caption{Final state of the embedding process in the second regime} \label{fig:R2}
\end{figure}

We now show that $G$ is indeed $(k-n^{\frac45},m)$-universal. Let $H$ be an arbitrary graph on $k-n^{\frac45}$ vertices and $m$ edges. Our task is to find an embedding of $H$ into $G$. 
\par First, embed the $\frac{n}{\log^3 n}$ vertices of highest degree from $H$ arbitrarily into $V_1$, and note that $n^{\frac45}$ vertices in $V_1$ remain free. Second, embed the next $n'$ vertices of highest degree into $V_2$. As the subgraph of $H$ induced by those vertices has less than $m$ edges, it can be embedded into $V_2$, as $G[V_2]$ is $(n',m)$-universal. 
\par Let $H'$ be the subgraph consisting of the remaining vertices of $H.$ Note that $|H'|=k-n^{\frac45}-\frac{n}{\log^3n}-n'=|V_3|$. We embed vertices of $H'$ one by one into $V_3$ and delete them from $H'$ until only $n^{\frac45}$ are left. For each $v\in H'$ let $S_v$ be the image set of all $v$'s neighbours in $H$ which are already embedded into $V_2\cup V_3$. At each step we argue that we can find a vertex $v\in H'$ and a free vertex $u\in V_3$ such that $u$ is a common neighbour of $S_v$ and we embed $v$ into $u$.\par 
Note that the size of each $S_v$ at each step is at most $\Delta_1=2m/n'=\frac{2\log n}{32\log(\frac{n\log n}{m})}\leq \frac{\log n}{2\log\frac{1}{p}}\leq \log n$, because $\Delta_1$ is an upper bound on the degree of vertices in $H'$. Note also that each vertex in $V_2\cup V_3$ can only be in at most $\Delta_2=2m/(\frac{n}{\log^3n})\leq \log^5n$ sets $S_v$, as vertices in $V_2\cup V_3$ cannot be images of one of the $\frac{n}{\log^3n}$ vertices of highest degree in $H$, so their degree in $H$ is at most $2m/(\frac{n}{\log^3n})$. Now we find a disjoint collection of $S_v$'s, by choosing them one by one and each time deleting all other sets which intersect the chosen one. Thus we get a collection of at least $\frac{|H'|}{\Delta_1\Delta_2+1}\ge n^{\frac34}$ disjoint sets $S_i$, since $|H'|\geq n^{\frac45}$. \par
To summarise we have found $n^{\frac34}$ disjoint sets $S_v$ each of size at most $\frac{\log n}{2\log(\frac{1}{p})}$. The set of remaining free vertices of $V_3$ has size $|H'| \ge n^{\frac45}$ and is disjoint from these $S_v$'s. Therefore, the domination property of $G$ implies that there is a free vertex in $V_3$ which is a common neighbour of all vertices in some set $S_v$. We embed $v$ into this vertex. We update each $S_v$ after every single embedding. When fewer than $n^{\frac45}$ vertices are left in $H'$ we embed them into the remaining free part of $V_1$ and are done. \par
We have shown the existence of a $k$-vertex graph $G$ which is $(k-n^{\frac45},m)$-universal. By making use of \Cref{lem:almost-spanning-to-spaning} and noting that $$k-\floor{ \frac{k(n-k)}{2m+k}}\leq k-\frac{n(n-k)}{9m}\leq k-\frac{n-k}{\log n}\le k-\frac{n}{\log^4 n} \leq k-n^{\frac45},$$ we conclude that by adding $n-k=\frac{n}{\log^3 n}$ full degree vertices to $G$ we get an $n$-vertex, $(n,m)$-universal graph. The number of edges of this graph is at most $$g(n',m)+\frac{17m^3}{n\log^3 n}+\frac{n^2}{\log^3n}\leq g(n',m)+\frac{25m^3}{n\log^3n}$$ as $n/2\leq m$, which finishes the proof.
\end{proof}

\vspace{-0.3cm}
\begin{cor}\label{cor:upper-bound-small}
For $n/2 \le m $ we have $$g(n,m) \le \OO\left( \frac{m^2}{\log ^2 m}\right).$$ 
\end{cor}

\vspace{-0.7cm}
\begin{proof}
If $n<\frac{16m}{ \log m}$ then the statement holds trivially since a graph on $n$ vertices has at most $\binom{n}{2}$ edges.

 For any fixed $m$ we will use the previous lemma and induction on $n$  to show 
 \begin{equation}\label{induction hypothesis}
 g(n,m) \le \frac{2^{22}m^2}{\log^2 m}\left(1-\frac{2m}{n \log n}\right)
 \end{equation}
 holds for all $n$ such that $\frac{16m}{\log m} \le n \le 2m$. The lower bound on $n$ implies that $n \geq \sqrt{m}$ and therefore $\log m \leq 2 \log n$.
 For the base of the induction consider all $n$ in our range such that $n \le \frac{2^{11}m}{\log m}$. In this case the inequality \eqref{induction hypothesis} holds, since $n \ge \frac{16m}{\log m}$ implies $\frac{2m}{n \log n}\le \frac12$ and $n^2 \le \frac{2^{22}m^2}{\log^2 m}$ so the RHS is at least $n^2/2$ and as before any graph on $n$ vertices has at most $\binom{n}{2}$ edges.

 We now proceed to the induction step. Let $2m \ge n \ge \frac{2^{11}m}{\log m}$ and assume that the statement holds for all smaller $n$ (but still larger than $\frac{16m}{\log m}$).
 Since our bounds on $n$ imply $\frac{n\log n}{2^{10}}\geq m \ge n/2$ we have that:
 \begin{align*}
     g(n,m)&\leq \frac{32m^3}{n\log^3 n}+g(n',m)
 \end{align*}
 where $n'=\frac{32m}{\log n}\cdot \log(\frac{n\log n}{m})$ is given by the previous lemma. By inequality \eqref{n'<n} we have $n'<n$ and from the definition $n'\ge \frac{32m}{\log n}\ge \frac{16m}{\log m}$ so we can apply the induction hypothesis to $n'$ to obtain
 \begin{align*}
     g(n,m)&\leq \frac{32m^3}{n\log^3 n}+\frac{2^{22}m^2}{\log ^2 m}\left(1-\frac{2m}{n'\log n'}\right)
     \\ &= \frac{2^{22}m^2}{\log^2 m}\left(1+\frac{m\log^2m}{2^{17}n\log^3n}-\frac{2m}{n'\log n'}\right)
     \end{align*}
 In order to finish the proof it is enough to show that:
 \[\frac{2m}{n'\log n'}-\frac{m\log^2m}{2^{17}n\log^3n}\geq \frac{2m}{n\log n}\]
 or equivalently  \[\frac{2\log n}{n'\log n'}-\frac{\log^2m}{2^{17}n\log^2n}\geq \frac{2}{n}.\]
Recall that $\log m \leq 2 \log n$ and $\log n' < \log n$ (since $n'<n$). Thus it is enough to show
\[\frac{2}{n'}\geq\frac{3}{n}.\]
which is true by inequality \eqref{n'<n}, so we are done.
\end{proof}

If one looks at what kind of graph this recursive argument builds, in each step it will add a few more vertices of full degree, in total at most $\OO(\frac{m^2}{n \log^2 m})$ since otherwise we would have used too many edges. The rest of the graph consists of several blocks which are initially small and have a large number (of randomly chosen) neighbours and progressively the blocks become bigger and bigger but they have less (randomly chosen) neighbours. The number of blocks we see is the number of times we needed to call upon the recursion and from definition of $n'$ we roughly have $\frac{ n' \log n'}{m} \approx \log \left(\frac{n \log n}{m} \right)$ where we used that $\log n' \approx \log n$ as throughout the argument we stay in the same range depending on $m$. What this means is that in each step we take a logarithm of the current value of $\frac{n \log n}{m}$ up until it reaches a constant. I.e. we need $\OO\left(\log^*\left(\frac{n \log n}{m} \right)\right)$\footnote{$\log^*x$ is the function defined as the number of times we need to apply the logarithm function to $x$ in order to get to $1$, applied to the number of atoms in the universe it evaluates to about $5$.} many steps. See \Cref{fig:Recursion unpacked} for an illustration of the constructed graph.

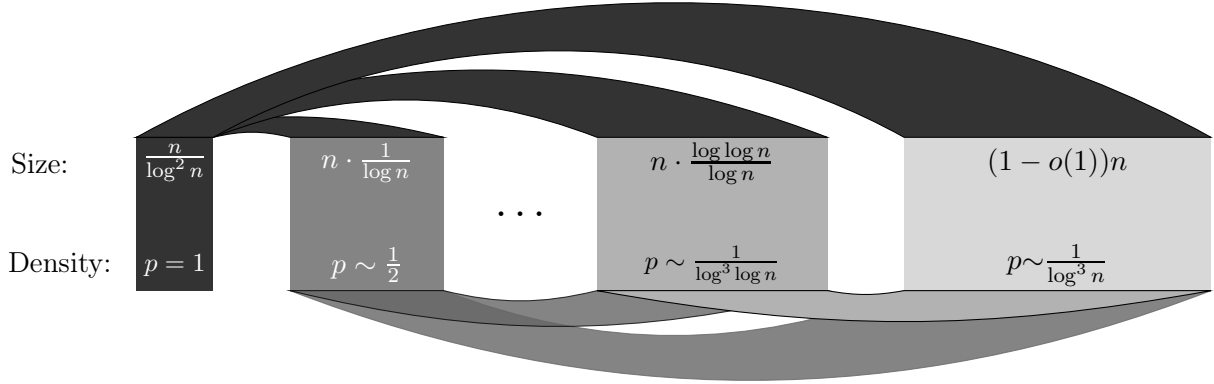
\begin{figure}
\centering
\begin{tikzpicture}[scale = 0.85, xscale=1.2, yscale=0.8]

\defPt {0}{0}{a1};
\defPt {1}{3}{a2};

\defPt {2}{0}{a3};
\defPt {4}{3}{a4};

\defPt {6}{0}{a5};
\defPt {9}{3}{a6};

\defPt {10}{0}{a7};
\defPt {14}{3}{a8};

\fill[opacity = 1, black!80] (a1) rectangle (a2);
\fill[ opacity = 0.8, black!60] (a3) rectangle (a4);
\fill[ opacity = 1, black!30] (a5) rectangle (a6);
\fill[opacity = 0.5, black!30] (a7) rectangle (a8);

\defPt{0}{3}{s}
\defPtm{($(a1)+(s)$)}{A1}
\defPtm{($(a3)+(s)$)}{A3}
\defPtm{($(a5)+(s)$)}{A5}
\defPtm{($(a7)+(s)$)}{A7}

\defPt{0}{-3}{s}
\defPtm{($(a2)+(s)$)}{A2}
\defPtm{($(a4)+(s)$)}{A4}
\defPtm{($(a6)+(s)$)}{A6}
\defPtm{($(a8)+(s)$)}{A8}


\filldraw[opacity=1, fill=black!80]
  (A1) to[bend left=20]
  (a4) --
  (A3) to[bend right=20]
  (a2) -- cycle;
  
  \filldraw[opacity=1, fill=black!80]
  (A1) to[bend left=30]
  (a6) --
  (A5) to[bend right=30]
  (a2) -- cycle;

  \filldraw[opacity=1, fill=black!80]
  (A1) to[bend left=40]
  (a8) --
  (A7) to[bend right=40]
  (a2) -- cycle;
 \filldraw[color=black!60, opacity=0.8, fill=black!60]
  (a3) to[bend right=30]
  (A8) --
  (a7) to[bend left=30]
  (A4) -- cycle;
  
  \filldraw[opacity=0.8, fill=black!60]
  (a3) to[bend right=20]
  (A6) --
  (a5) to[bend left=20]
  (A4) -- cycle;

\filldraw[opacity=1, fill=black!30]
  (a5) to[bend right=15]
  (A8) --
  (a7) to[bend left=15]
  (A6) -- cycle;

\defPt{-1}{0}{s}
\defPtm{($(s)+0.5*(a3)+0.5*(A7)$)}{tackice}
 \node[scale=1.5] (tackice) at (tackice) {$\ldots$};
\defPt{0}{-0.5}{s}
\defPtm{($0.5*(A1)+0.5*(a2)+(s)$)}{S1}
 \node[color=white, scale=1] (S1) at (S1) {$\frac{n}{\log^2n}$};
 \node[scale=1] () at ($(S1)-(1.76,0)$) {Size:};
\defPtm{($0.5*(A3)+0.5*(a4)+(s)$)}{S2}
 \node[color=white, scale=1] (S2) at (S2) {$n\cdot \frac{1}{\log n}$};
\defPtm{($0.5*(A5)+0.5*(a6)+(s)$)}{S3}
 \node[ scale=1] (S3) at (S3) {$n \cdot \frac{\log\log n}{\log n}$};
\defPtm{($0.5*(A7)+0.5*(a8)+(s)$)}{S4}
 \node[ scale=1] (S4) at (S4) {$(1-o(1))n$};

\defPt{0}{0.5}{s}
\defPtm{($0.5*(a1)+0.5*(A2)+(s)$)}{S1}
 \node[color=white, scale=0.9] (S1) at (S1) {$p=1$};
 \node[scale=1] () at ($(S1)-(1.5,0)$) {Density:};
\defPtm{($0.5*(a3)+0.5*(A4)+(s)$)}{S2}
 \node[color=white, scale=1] (S2) at (S2) {$p\sim \frac{1}{2}$};
\defPtm{($0.5*(a5)+0.5*(A6)+(s)$)}{S3}
 \node[ scale=0.9] (S3) at (S3) {$p\sim \frac{1}{\log^3 \log n}$};
\defPtm{($0.5*(a7)+0.5*(A8)+(s)$)}{S4}
 \node[ scale=1] (S4) at (S4) {$p{\sim} \frac{1}{\log^3 n}$};

 \end{tikzpicture}
\caption{The $(n,n)$-universal graph constructed by the recursion. Each vertex in a block has an edge towards another vertex in its own block or a vertex in a subsequent block with probability $p$. The density of the random edges, depicted by different shades of gray, decreases as the size of the block increases. } \label{fig:Recursion unpacked}
\end{figure}

\subsection{Completing the picture}
We are now ready to combine our results to show our main theorem. 
\begin{proof}[ of \Cref{thm:main2}]
The lower bounds follow directly from \Cref{lem:counting}.
In \Cref{cor:upper-bound-small} we have proven the desired upper bound for $g(n,m)$ when $n/2\leq m\leq n\log n$. \Cref{thm:upper-bound-big} gives us the bound when $n\log n<m<n^{3/2-\varepsilon}$. 

The remaining case is when $m<n/2$. Notice that in this case $g(n,m)\leq g(2m,m)=\Theta\left(\frac{m^2}{\log^2m}\right)$ since any graph with $e$ edges has at most $2m$ non-isolated vertices.
\end{proof}

\begin{rem}
We have shown that the function $g(n,m)$ exhibits different behaviour in the two regimes. Namely, when $m=o(n\log n)$ then $g(n,m)=\Theta\left(\frac{m^2}{\log^2 m}\right)=o(n^2)$ while if $m=\omega(n\log n)$ we have that $g(n,m)=\binom{n}{2}-\Theta\left(\frac{n^3\log n}{m}\right)=(1-o(1))\binom{n}{2}$. From our arguments one can see that as we transition from one regime to another both bounds become quadratic in $n$. The following corollary shows that $m=\Theta (n\log n)$ is precisely the transitioning point between the two behaviours.
\end{rem}

\begin{cor}
Let $\mu>0$ be a positive constant. Then there exist positive constants $0<c_1,c_2<1$ such that $c_1\binom{n}{2}\leq g(n,\mu n\log n)\leq  c_2\binom{n}{2}$
for all positive integers $n$.
\end{cor}
\begin{proof}

The lower bound follows from \Cref{lem:counting}. For the upper bound we have the following cases:
\begin{itemize}
    \item  When $m>n\log n$ the claim follows from \Cref{thm:upper-bound-big}.
    \item  When $m<\frac{n\log n}{2^{12}}$ the claim follows from \eqref{induction hypothesis}.
    \item If $\frac{n\log n}{2^{12}}\leq e\leq  n\log n$ then we use the inequality $g(n,x)\leq 2g(n/2,x)+(n/2)^2$ iterating it a constant number of times until we are able to use \Cref{thm:upper-bound-big}. This inequality holds as one can construct an $(n,x)$-universal graph by taking two disjoint copies of a $(n/2,x)$-universal graph and making every two vertices from different copies adjacent.
\end{itemize}
\vspace{-0.8 cm}
\end{proof}

\section{Concluding Remarks}

In this paper we complete the study of the behaviour of $f(n,e)$ defined as the maximum size of an $(n,e)$-unavoidable graph.

As already mentioned in the introduction, in order to get a lower bound on $f(n,e)$ we want to find a graph $H$ with $f(n,e)$ edges and Tur\'an number $\ex(n,H)\le e.$ Inverting this statement we want to find graphs with fixed number of edges which minimise the Tur\'an number. The universal graph we used in the proof of Theorem \ref{thm:upper-bound-big} consisted of a collection of vertices joined to every other vertex and a random graph on the remainder. By transferring this result into unavoidability language, as discussed at the end of \Cref{The first regime}, we obtain that whp the appropriate random graph $\G(n/4,q)$ is the desired minimiser of the Tur\'an number, provided $1/2>q>n^{-1/2+\eps}$. Being a bit more careful with our estimates one can even obtain an almost spanning version of this result. Namely we can get:

\begin{thm}
Let $\eps>0$. There is a $\delta>0$ such that for $G\sim  \G\big((1-\varepsilon)n, q\big)$ we have whp that $\ex(n,G)=\left(1-\Theta\left(\frac{\log n}{nq}\right)\right)\binom{n}{2},$ provided $n^{-1/2+\eps}<q<\delta$.
\end{thm}

Note that for this range of $q$ the chromatic number of the above random graph $G$ whp satisfies $\chi(G) =\Theta\left(\frac{nq}{\log n}\right)$.
Interestingly, this shows that $\ex(n,G)$ behaves essentially the same as we would expect from the Erd\H{o}s-Stone-Simonovits theorem \cite{ES,ESS}, i.e. $\ex(n,G)=\left(1-\frac{\Theta(1)}{\chi(G)}\right)\binom{n}{2}$, despite the fact that $G$ is almost spanning! In contrast, in order for the Erd\H{o}s-Stone-Simonovits theorem to apply, the size of the host graph is required to be significantly larger than the graph being embedded. The question of the exact requirement on the parameters was considered by Bollob\'as \cite{bollobas} and Chv\'atal and Szemer\'edi \cite{chvatal-szemeredi} who showed that the best one can hope for in general is that the Erd\H{o}s-Stone-Simonovits theorem holds for graphs of order $O(\log n)$, even in our approximate sense.

It could be interesting to determine whether the above theorem extends for values of $q$ smaller than $n^{-1/2}.$ The main obstacle for our argument is that our current embedding strategy requires too strong a domination property, which in particular is no longer satisfied by $\G(n,q).$ In light of this it might be interesting to try to find a weaker version of our property which would suffice for our embedding argument but is still satisfied by the sparser random graphs. Another possible benefit of such a weaker property is that it could possibly allow one to construct explicit universal graphs which one could use for our argument and answer a question of Alon and Asodi \cite{alon2002sparse} and later Hetterich, Parczyk and Person\cite{hetterich}.

Several bounded degree analogues of our universality problem arise quite naturally. Alon and Capalbo \cite{alon2008optimal} show that the minimal number of edges in a graph which is universal for the family $\mathcal{E}(n,d)$ of $n$-vertex graphs of maximum degree at most a constant $d$ is $\Theta\left(n^{2-2/d}\right)$. This is asymptotically very different from $g(n,dn)=\Theta\left(\frac{n^2}{\log^2n}\right)$ which we get from Theorem \ref{thm:main2}. 
We determine $g(n,dn)$ up to a constant factor for any values of $n$ and $d$, even if we allow $d$ to depend on $n$. However, very little seems to be known about the above bounded degree problem if one allows $d$ to grow with $n$. In particular, what is the minimal number of edges in an $\mathcal{E}(n,d)$-universal graph when $d$ is allowed to depend on $n$. Another, possibly an even closer analogue, is what happens if we consider the spanning variant of this problem. So, what is the smallest number of edges in a graph on exactly $n$ vertices which is $\mathcal{E}(n,d)$-universal, where $d$ is allowed to depend on $n$?


\textbf{Acknowledgments.} The authors would like to thank Tuan Tran for bringing unavoidability problems to our attention. We are also grateful to the anonymous referees for their careful reading of the paper and many useful suggestions.


\providecommand{\bysame}{\leavevmode\hbox to3em{\hrulefill}\thinspace}
\providecommand{\MR}[1]{}
\providecommand{\MRhref}[2]{%
  \href{http://www.ams.org/mathscinet-getitem?mr=#1}{#2}
}
\providecommand{\href}[2]{#2}

\bibliographystyle{amsplain}

\providecommand{\bysame}{\leavevmode\hbox to3em{\hrulefill}\thinspace}
\providecommand{\MR}{\relax\ifhmode\unskip\space\fi MR }
\providecommand{\MRhref}[2]{%
  \href{http://www.ams.org/mathscinet-getitem?mr=#1}{#2}
}
\providecommand{\href}[2]{#2}

\end{document}